\documentclass[12pt]{amsart}
\usepackage{amscd}      
\usepackage{amssymb}
\usepackage{amsmath, amsthm, graphics}
\usepackage{xypic}      
\LaTeXdiagrams          
\usepackage[all]{xy}
\xyoption{2cell} \UseAllTwocells \xyoption{frame} \CompileMatrices
\allowdisplaybreaks[3]

\usepackage{latexsym}
\usepackage{epsfig}
\usepackage{amsfonts}
\usepackage{enumerate}
\usepackage{times}

\newtheorem{prop}{Proposition}[section]

\newtheorem{thm}[prop]{Theorem}

\newtheorem{question}[prop]{Question}

\newtheorem{theorem}{Theorem}[section]

\newtheorem{example}[theorem]{Example}

\theoremstyle{remark}

\theoremstyle{remark}

\numberwithin{equation}{section}

\newcommand{\com}{\mathbb{C}}

\def\<{\left\langle}
\def\>{\right\rangle}

\newcommand{\bj}{\mathbf{j}}

\newcommand{\complex}{{\mathbb C}}

\newcommand{\End}{\operatorname{End}}
\newcommand{\Ad}{\operatorname{Ad}}

\begin{document}
\title{Conjugacy classes and characters for extensions of finite groups}
\author[Tang]{Xiang Tang}
\address{Department of Mathematics\\ Washington University\\ St. Louis\\ MO 63130\\ USA}
\email{xtang@math.wustl.edu}

\author[Tseng]{Hsian-Hua Tseng}
\address{Department of Mathematics\\ Ohio State University\\ 100 Math Tower, 231 West 18th Ave. \\ Columbus \\ OH 43210\\ USA}
\email{hhtseng@math.ohio-state.edu}

\date{\today}

\begin{abstract}
Let $H$ be an extension of a finite group $Q$ by a finite group $G$. Inspired by the results of duality theorems for \'etale gerbes on orbifolds, we describe the number of conjugacy classes of $H$ that maps to the same conjugacy class of $Q$.  Furthermore, we prove a generalization of the orthogonality relation between characters of $G$. 
\end{abstract}

\maketitle

\section{Introduction}

Extensions of finite groups play an important role in the theory of finite groups. For example, the composition serious of a finite group $H$ consists of a sequence of subgroups $H_i$ 
\[
1=H_0\lhd H_1\lhd H_2\lhd \cdots\lhd H_n=H,
\] 
such that  $H_i$ is a strict normal subgroup of $H_{i+1}$ with a simple quotient group $H_{i+1}/H_i$, for $i=0, \cdots, n-1$. Therefore, with the classification theorem of finite simple groups, the study of extensions of finite
groups would describe and classify all finite groups. 

The structure of extensions of finite groups has been  studied for a long time, see \cite{Schreier}. In this paper, we look at extensions of finite groups from a geometric point of view. A finite group $G$ is a groupoid with one unit. In the language of stacks \cite{be-xu}, such a group(oid) corresponds to the classifying stack $BG$ of principal $G$-bundles. An extension of a finite group $Q$ by a finite group $G$
\[
1\longrightarrow G\longrightarrow H\longrightarrow Q\longrightarrow 1
\]
is equivalent to a $G$-gerbe 
\[
BH\longrightarrow BQ, 
\] 
a bundle of $BG$ over $BQ$, c.f. \cite{la-st-xu}. 

Our study of extensions of finite groups is motivated by a conjecture in Mathematical physics \cite{hel-hen-pan-sh}.  Let $\widehat{G}$ be the finite set of isomorphism classes of irreducible unitary representations of $G$. The above extension $H$ of $Q$ by $G$ gives a natural action of $Q$ on $\widehat{G}$. Consider the transformation groupoid $\widehat{G}\rtimes Q\rightrightarrows \widehat{G}$. There is a canonical class $c$ in $H^{2}(\widehat{G}\rtimes Q, U(1))$ associated to the extension $H$.  The decomposition conjecture in \cite{hel-hen-pan-sh} suggests that the geometry of a $G$-gerbe associated to the extension $H$ is equivalent to the geometry of the orbifold associated to the groupoid $\widehat{G}\rtimes Q$ twisted by $c$. We studied this   conjecture in \cite{TT} from the view point of noncommutative geometry. In particular, we proved that the group algebra of $H$ is Morita equivalent the $c$-twisted groupoid algebra of $\widehat{G}\rtimes Q$. The detail of this is reviewed in Section \ref{sec:groupalgebra}.

In this short note, we present two results from our analysis of the structure of $\complex H$. One result concerns the relations between conjugacy classes of $H$ and $Q$, see Section \ref{subsec:conjugacy}. The other result concerns a generalized orthogonality relation between characters of $G$, see Section \ref{subsec:orthogonality}. 

To the best of our knowledge, the results in this paper are new. We would like to thank I. M. Isaacs for discussions related to Question \ref{counting_question}. Tang's research is partially supported by NSF grant 0900985, and NSA grant H96230-13-1-02. Tseng's research is partially supported by by NSF grant DMS-0757722 and Simons Foundation collaboration grant.

\section{Group algebras of finite group extensions}\label{sec:groupalgebra}
Consider an extension of finite groups as in
\begin{equation}\label{eq:extension_reproduced}
1\longrightarrow G\stackrel{i}{\longrightarrow} H
\stackrel{j}{\longrightarrow} Q\longrightarrow 1.
\end{equation}
As part of our study of gerbe duality, the structure of the group algebra $\complex H$ is analyzed in \cite{TT}. We briefly recall the results.

Choose a section $s:Q\rightarrow H$ of $j:H\to Q$ above such that $j\circ s=id$, and $s(1)=1$. Since $G$ and $Q$ are finite groups, such a section $s$ always exists. For $q_1,q_2\in Q$ define $\tau(q_1, q_2):=s(q_1)s(q_2)s(q_1q_2)^{-1}$. It is easy to see that $\tau(q_1,q_2)\in \ker(j)=G$, so we obtain 
$$\tau:Q\times Q\rightarrow G.$$
Clearly $\tau$ is trivial (i.e. $\tau(-,-)=1$) if and only if $s: Q\to H$ is a group homomorphism, which in turn is equivalent to the extension (\ref{eq:extension_reproduced}) being a split extension.

The definition of $\tau$ may be written as
\begin{equation}
\label{eq:tau-def} s(q_1)s(q_2)=\tau(q_1,q_2)s(q_1q_2).
\end{equation}
By associativity, we have
$(s(q_1)s(q_2))s(q_3)=s(q_1)(s(q_2)s(q_3))$. It follows that 
\begin{equation}\label{eq:tau-cocycle}
\tau(q_1,q_2)\tau(q_1q_2,q_3)=s(q_1)\tau(q_2,q_3)s(q_1)^{-1}\tau(q_1,q_2q_3).
\end{equation}

Given the section $s$, we can define a set-theoretic bijection between $H$ and $G\times Q$:
$$\alpha: H\rightarrow G\times Q,\quad \alpha(h):=(hs(j(h))^{-1}, j(h)).$$ 
The inverse of $\alpha$ is $$G\times Q\to H, \quad (g,q) \mapsto i(g)s(q).$$ 
The group structure on $H$ induces a new group structure $\cdot$ on $G\times Q$ via $\alpha$. This group structure is given by
\begin{equation}\label{eq:twisted-prod-gp}
(g_1,q_1)\cdot (g_2,q_2)=(g_1\Ad_{s(q_1)}(g_2)\tau(q_1,q_2),q_1q_2).
\end{equation}
Here $\Ad_{h}(\cdot)$ denotes the conjugation action of an element $h\in H$ on $G$, which is an automorphism of $G$ because $G$ is normal in $H$. Denote by 
$$G\rtimes_{s,\tau}Q$$ the set $G\times Q$ with the group structure given by (\ref{eq:twisted-prod-gp}). The definition implies that $\alpha$ is a group isomorphism: $$\alpha: H \to G\rtimes_{s,\tau}Q.$$ It is easy to check that different choices of the section $s$ yield isomorphic groups $G\rtimes_{s,\tau}Q$.

The group isomorphism $\alpha$ naturally induces an isomorphism of group algebras 
$$\alpha: \complex H \overset{\simeq}{\longrightarrow} \complex (G\rtimes_{s,\tau}Q).$$
Given $s$ and $\tau$, we let an element $q\in Q$ act on $\complex G$ by conjugation by $s(q)$. This does not give an action of $Q$ on $\complex G$, and the failure of this to be an action is governed by $\tau$. In other words, this defines a $\tau$-twisted action of $Q$ on $\complex G$. Hence the group algebra $\complex(G\rtimes_{s,\tau}Q)$ can be written as a twisted crossed product algebra $\complex G\rtimes _{s,\tau}Q$.

Let $\widehat{G}$ be the set of isomorphism classes of irreducible complex linear representations of $G$. Furthermore, for every element $[\rho]$ in $\widehat{G}$, we {\em choose} an irreducible representation in the class $[\rho]$ denoted by $$\rho:G\to \End(V_\rho),$$ where $V_\rho$ is a certain finite dimensional $\complex$-vector space. The group algebra $\complex G$ is isomorphic to a direct sum of matrix algebra $\oplus_{[\rho]\in \widehat{G}}\End(V_\rho)$:
$$\beta: \complex G\overset{\simeq}{\longrightarrow} \oplus_{[\rho]\in \widehat{G}}\End(V_\rho), \quad g\mapsto (\rho(g))_{[\rho]\in \widehat{G}}.$$
This is well-known, see e.g. \cite[Proposition 3.29]{fu-ha}. 

Next we define an action of $Q$ on $\widehat{G}$.  Let $\rho: G\to \End(V_\rho)$ be a $\complex$-linear representation of $G$. Given $q\in Q$, we obtain another $G$ representation $\tilde{\rho}$ defined by $$G\ni g\mapsto \rho(\Ad_{s(q)}(g)).$$ It is easy to see that $\tilde{\rho}$ is irreducible if and only if $\rho$ is. If $s': Q\to H$ is another section of $j$, then we have $\rho\circ \Ad_{s(q)}=\rho\circ \Ad_{s'(q)}\Ad_{s'(q)^{-1}s(q)}$. Since $s'(q)^{-1}s(q)\in G$, $\Ad_{s'(q)^{-1}s(q)}$ is an inner automorphism of $G$. Hence $\rho\circ \Ad_{s(q)}$ and $\rho\circ \Ad_{s'(q)}$ are isomorphic $G$-representations. Therefore the assignment $(q, \rho)\mapsto \tilde{\rho}$ yields a right $Q$-action on $\widehat{G}$;
namely, $q\in Q$ sends the class $[\rho]\in \widehat{G}$ to the class $[\tilde{\rho}]\in \widehat{G}$. For notational convenience, we write this right action as a left action. We denote the image of the isomorphism class $[\rho]\in \widehat{G}$ under the action by $q$ by $q([\rho])$. By abuse of notation, we denote the chosen irreducible $G$-representation that represents the class $q([\rho])$ also by $q([\rho]): G\to \End(V_{q([\rho])})$. Let $$\widehat{G}\rtimes Q:=(\widehat{G}\times Q \rightrightarrows\widehat{G})$$ be the groupoid associated to this $Q$-action on $\widehat{G}$.

By construction, the representation $q([\rho]): G\to \End(V_{q([\rho])})$ is equivalent to the representation $\tilde{\rho}: G\to \End(V_\rho)$ defined by $g\mapsto \rho(\Ad_{s(q)}(g))$. Therefore there exists a $\complex$-linear isomorphism,
$$T_q^{[\rho]}: V_{\rho}\to V_{q([\rho])},$$ that intertwines the two representations, namely
$$\rho(\Ad_{s(q)}(g))={T^{[\rho]}_q}^{-1}\circ q([\rho])(g)\circ T^{[\rho]}_q.$$ We may choose $T^{[\rho]}_1$ to be the identity map on $V_\rho$. It can be shown that there are constants
$c^{[\rho]}(q_1, q_2)$ such that $T^{q_1([\rho])}_{q_2}\circ
T^{[\rho]}_{q_1}\circ \rho(\tau(q_1,q_2))\circ
{T^{[\rho]}_{q_1q_2}}^{-1}$ is $c^{[\rho]}(q_1, q_2)$ times the
identity map. In other words,
\begin{equation}
\label{eq:dfn-c}T^{q_1([\rho])}_{q_2}\circ
T^{[\rho]}_{q_1}=c^{[\rho]}(q_1,q_2)T^{[\rho]}_{q_1q_2}\rho(\tau(q_1,q_2))^{-1}.
\end{equation}
Since the collection $\{\rho\}$ consists of unitary representations, the isomorphisms $T^{[\rho]}_q$ can also be chosen to be unitary. Therefore, $c^{[\rho]}(q_1,q_2)$ actually takes value in $U(1)$. By \cite[Proposition 3.1]{TT}, The function 
$$c:\widehat{G}\times Q\times Q\to U(1),\quad ([\rho], q_1,q_2)\mapsto c^{[\rho]}(q_1,q_2)$$ is a 2-cocycle on the groupoid $\widehat{G}\rtimes Q$ such that
$c^{[\rho]}(1,q)=c^{[\rho]}(q,1)=1$ for any $[\rho]\in \widehat{G},
q\in Q$. The cohomology class defined by $c$ is independent of the
choices of the section $s$ and the operator $T^{[\rho]}_q$.

Let $$C(\widehat{G}\rtimes Q,c),$$ be the {\em twisted groupoid algebra} associated to the cocycle $c$ on $\widehat{G}\rtimes Q$. We explain the definition of $C(\widehat{G}\rtimes Q,
c)$ and refer the readers to \cite{tu-la-xu} for more details. By definition $C(\widehat{G}\rtimes Q, c)$ is the set of $C(\widehat{G})$-valued functions on $Q$, i.e., $\complex$-valued functions on $\widehat{G}\times Q$. By abuse of notation, for $([\rho],q)\in \widehat{G}\times Q$ we also denote by $([\rho], q)$ the function on $\widehat{G}\times Q$ which takes value $1$ at $([\rho],q)$ and $0$ elsewhere. The collection $\{([\rho], q)\}$ of functions on $\widehat{G}\times Q$ forms an additive basis of
$C(\widehat{G}\rtimes Q, c)$. The set $C(\widehat{G}\rtimes Q, c)$ is endowed with a product structure defined by
\[
([\rho],q)\circ([\rho'],q')=\left\{\begin{array}{ll}c^{[\rho]}(q,q')([\rho],
qq')&\ \text{if\ }[\rho']=q([\rho])\\ 0&\
\text{otherwise}\end{array}\right. .
\]
The cocycle condition of $c$ implies that this product is associative.

Let $\oplus_{[\rho]\in \widehat{G}}\End(V_\rho)\otimes \complex Q$ be the $\complex$-vector space spanned by elements of the form $(x_\rho,
q)$, where $x_\rho$ is an element in $\End(V_\rho)$ with $[\rho]\in
\widehat{G}$ and $q\in Q$. We equip this space with a product $\circ$ defined as follows:
\[
(x_{\rho_1},q_1)\circ (\tilde{x}_{\rho_2},q_2):=
\left\{
\begin{array}{ll}
(x_{\rho_1}{T^{[\rho_1]}_{q_1}}^{-1}\tilde{x}_{q_1([\rho_1])}T^{[\rho_1]}_{q_1}\rho_1(\tau(q_1,q_2)),
q_1q_2),&
\text{if}\ [\rho_2]=q_1([\rho_1]),\\
0&\text{otherwise}.
\end{array}
\right.
\]
Let $$\oplus_{[\rho]}\End(V_{\rho})\rtimes_{T,\tau}Q$$ be the space $\oplus_{[\rho]\in \widehat{G}}\End(V_\rho)\otimes \complex Q$ with the product $\circ$ defined above. We call this the twisted crossed product algebra. This algebra plays an important role in the following structure result on the group algebra $\complex H$:
\begin{prop}[\cite{TT}, Proposition 3.2]\label{prop:matrix-coeff-algebra}
The map $$\kappa: G\times Q\ni (g,q)\mapsto \sum_{[\rho]\in \widehat{G}}(\rho(g),q)$$ defines an algebra isomorphism from the group algebra $\complex G\rtimes_{s,\tau}Q$ to the twisted crossed product algebra $\oplus_{[\rho]}\End(V_{\rho})\rtimes_{T,\tau}Q$. Hence, $$\kappa\circ\alpha:\complex H\to
\oplus_{[\rho]}\End(V_{\rho})\rtimes_{T,\tau}Q$$ is an algebra isomorphism.
\end{prop}

Proposition \ref{prop:matrix-coeff-algebra} is used in \cite[Section 3.2]{TT} to prove the following structure result of $\complex H$:
\begin{thm}[\cite{TT}, Theorem 3.1] \label{thm:local-mackey}
The group algebra $\complex H$ is Morita equivalent to the twisted groupoid algebra $C(\widehat{G}\rtimes Q, c)$.
\end{thm}

We remark that the proof of Theorem \ref{thm:local-mackey} is done by explicitly constructing Morita equivalence bimodules between the two algebras.

Since $j:H\to Q$ is a surjective group homomorphism, $j$ induces a surjective homomorphism of algebras from $\complex H$ to $\complex Q$. It is well-known that the center of $\complex Q$ has a canonical additive basis indexed by the conjugacy classes of $Q$. This decomposition of the center $Z(\complex Q)$ and the surjection $\complex H\to \complex Q$ implies that the center of $\complex H$, as a vector space, decomposes into a
direct sum of subspaces $Z(\complex H)_{\<q\>}$ indexed by conjugacy
classes $\<q\>$ of $Q$,
$$Z(\complex H)=\bigoplus_{\<q\>\subset Q} Z(\complex H)_{\<q\>}.$$
As shown in \cite[Section 3.2]{TT}, the center $Z(C(\widehat{G}\rtimes Q, c))$
decomposes into a direct sum of subspaces $Z(C(\widehat{G}\rtimes Q,
c))_{\<q\>}$ indexed by conjugacy classes of $Q$,
$$Z(C(\widehat{G}\rtimes Q, c))=\bigoplus_{\<q\>\subset Q}Z(C(\widehat{G}\rtimes Q,
c))_{\<q\>}.$$

The explicit Morita equivalence bimodules in the proof of Theorem \ref{thm:local-mackey} yield an algebra isomorphism from the center of $\complex H$ to the center of $C(\widehat{G}\rtimes Q, c)$, which we denote by $I$.

\begin{prop}[\cite{TT}, Proposition 3.4]\label{prop:conjugacy}
The isomorphism $$I:Z(\complex H)\to Z(C(\widehat{G}\rtimes Q,c))$$ is compatible with the decompositions into subspaces indexed by conjugacy classes of $Q$, i.e., $I$ is an isomorphism from $Z(\complex H)_{\<q\>}$ to $Z(C(\widehat{G}\rtimes Q, c))_{\<q\>}$.
\end{prop}

In the rest of this paper, we discuss some group-theoretic applications of
our analysis of the group algebra $\com H$.

\section{Counting conjugacy classes in group
extensions}\label{subsec:conjugacy}

Let $j: H\to Q$ be a surjective homomorphism of finite groups. Let
$\<q\>\subset Q$ be a conjugacy class of $Q$. The pre-image
$$j^{-1}(\<q\>)\subset H$$ may be partitioned into a disjoint union
of conjugacy classes of $H$. It is natural to ask the following:
\begin{question}\label{counting_question}
How many conjugacy classes of $H$ are contained in $j^{-1}(\<q\>)$?
\end{question}
In this Section, we discuss an answer to this question.

Let $G$ be the kernel of $j: H\to Q$. Then we are in the situation
of the exact sequence (\ref{eq:extension_reproduced}). The
homomorphism $j: H\to Q$ induces a surjective homomorphism $\bj:
\com H\to \com Q$ between group algebras. This, in turn, induces a 
homomorphism $\bj :Z(\com H)\to Z(\com Q)$ between
centers. The centers $Z(\com H)$ and $Z(\com Q)$, viewed as vector
spaces, admit natural bases, $\{1_{\<h\>}\}\subset Z(\com H)$ and
$\{1_{\<q\>}\}\subset Z(\com Q)$, indexed by conjugacy classes. 
These bases satisfy the requirement that if $j(\<h\>)=\<q\>$, then
$\bj(1_{\<h\>})\in \mathbb{N} 1_{\<q\>}$. As $j(\<s(q)\>)=\<q\>$, the map $j:Z(\com H)\to Z(\com Q)$ is surjective. Let
$$Z(\com H)_{\<q\>}:=\bigoplus_{\<h\>\subset j^{-1}(\<q\>)} \com
1_{\<h\>}.$$ By construction, the dimension $\text{dim}\, Z(\com
H)_{\<q\>}$ is the number of conjugacy classes of $H$ that are
contained in $j^{-1}(\<q\>)$. By Proposition \ref{prop:conjugacy},
the isomorphism $I: Z(\com H)\to Z(C(\widehat{G}\rtimes Q,c))$
restricts to an additive isomorphism
$$Z(\com H)_{\<q\>}\simeq Z(C(\widehat{G}\rtimes Q, c))_{\<q\>}.$$
Clearly, the answer to Question \ref{counting_question} is the dimension $\text{dim}\, Z(C(\widehat{G}\rtimes Q, c))_{\<q\>}$, which we now compute.

Let $\widehat{G}^q\subset \widehat{G}$ be the subset consisting of
elements fixed by $q\in Q$. Let $C(q)\subset Q$ be the centralizer
subgroup of $q$. Then, by \cite{ru1}, we have that $Z(C(\widehat{G}\rtimes Q,
c))_{\<q\>}$ is additively isomorphic to the $c$-twisted orbifold
cohomology $H_{orb}^\bullet([\widehat{G}^q/C(q)], c)$. Decompose
$\widehat{G}^q$ into a disjoint union of $C(q)$-orbits:
\begin{equation}\label{orbit_decomp}
\widehat{G}^q=\coprod_i O_i.
\end{equation}
For each $C(q)$-orbit $O_i$, pick a representative $[\rho_i]$ and
denote by $Q_i:=\text{Stab}_{C(q)}([\rho_i])\subset C(q)$ the
stabilizer subgroup of $[\rho_i]$. Consider the homomorphism
$$\gamma_{-,q}^{[\rho_i]}: C(q)\to U(1), \quad C(q)\ni q_1\mapsto
\gamma_{q_1,q}^{[\rho_i]}:=c^{[\rho_i]}(q_1,q)c^{[\rho_i]}(q,q_1)^{-1}.$$
Here, $c^{[\rho]}(-,-)$ is the cocycle defined in (\ref{eq:dfn-c}).
It follows from (\ref{orbit_decomp}) that
$$H_{orb}^\bullet([\widehat{G}^q/C(q)], c)\simeq \bigoplus H_{orb}^\bullet(BQ_i, c).$$
By \cite[Example 6.4]{ru1}, we have that $H_{orb}^\bullet(BQ_i, c)=\com$ if the
following condition holds:
\begin{equation}\label{condition:gamma=1}
\gamma_{q_1,q}^{[\rho_i]}=1\,\, \text{ for all } q_1\in Q_i.
\end{equation}
Moreover, if (\ref{condition:gamma=1}) does not hold, then
$H_{orb}^\bullet(BQ_i, c)=0$. It follows that $\text{dim}\,
Z(C(\widehat{G}\rtimes Q, c))_{\<q\>}$ is equal to
\[
\#\{O_i=C(q)\text{-orbit of }\widehat{G}^q| \text{ there exists }
[\rho_i]\in O_i \text{ s.t. } \gamma_{q_1,q}^{[\rho_i]}=1\,\,
\text{for all } q_1\in Q_i=\text{Stab}_{C(q)}([\rho_i])\}.
\]

In summary, we have obtained the following theorem as an answer to Question \ref{counting_question}.
\begin{theorem}\label{thm:conj-general} Let $H=G\rtimes_{s, \tau} Q$ be an extension of $Q$ by $G$. Consider the canonical quotient map $j:H\to Q$. For $q\in Q$, the number of conjugacy classes of $H$ that is mapped to the conjugacy class $\<q\>$ of $Q$ is equal to   
\begin{equation}\label{count_answer1}
\#\{O_i=C(q)\text{-orbit of }\widehat{G}^q| \text{ there exists }
[\rho_i]\in O_i \text{ s.t. } \gamma_{q_1,q}^{[\rho_i]}=1\,\,
\text{for all } q_1\in Q_i=\text{Stab}_{C(q)}([\rho_i])\}.
\end{equation}

\end{theorem}

In the following, we discuss a few special cases of Theorem \ref{thm:conj-general}. 
\begin{example}
If the group $G$ is abelian, then all irreducible representations of
$G$ are $1$-dimensional, and all intertwiners
 in (\ref{eq:dfn-c})
can be taken to be the identity. In this case, (\ref{count_answer1})
can be simplified to
\begin{equation}\label{count_answer2}
\begin{split}
\#\{O_i=C(q)\text{-orbit of }\widehat{G}^q| &\text{ there exists }[\rho_i]\in O_i, \\
&\text{s.t. }\rho_i(\tau(q_1,q)\tau(q,q_1)^{-1})=1\,\, \text{for all } q_1\in Q_i=\text{Stab}_{C(q)}([\rho_i])\}.
\end{split}
\end{equation}
\end{example}

\begin{example}
If the group $G$ is abelian and $H$ is a semi-direct product of $G$
and $Q$, then the cocycle $\tau(-,-)$ can be taken to be trivial. In
this case, (\ref{count_answer1}) can be simplified to
\begin{eqnarray}\label{count_answer2.5}
&\#\{C(q)\text{-orbit of }\widehat{G}^q\}.
\end{eqnarray}
\end{example}

\begin{example}
If the $Q$-action on $\widehat{G}$ is trivial\footnote{Equivalently,
this means that the band of the gerbe $BH\to BQ$ is trivial.}, then
$\widehat{G}^q=\widehat{G}$, and all intertwiners in
(\ref{eq:dfn-c}) can be taken to be the identity. In this case,
(\ref{count_answer1}) can be simplified to
\begin{equation}\label{count_answer3}
\begin{split}
\#\{[\rho]&=\text{isomorphism class of irreducible }
G\text{-representations}|\\
&\hspace{4cm}\rho(\tau(q_1,q)\tau(q,q_1)^{-1})=1\,\, \text{for all
}q_1\in C(q)\}.
\end{split}
\end{equation}
\end{example}


\section{An orthogonality relation of characters}\label{subsec:orthogonality} 
The material in this
Section is inspired by the proof of the orthogonality relation
given in \cite[Chapter 2, Section 12]{Be_Zh}. Using Proposition \ref{prop:matrix-coeff-algebra}, we prove  a generalization of the orthogonality
relation between characters of $G$. For $h\in H$, write the centralizer subgroup of $h$ by $C_H(h)$, and the number of elements in $C_H(h)$ by $|C_H(h)|$.

\begin{theorem}\label{thm:character} Let $H=G\rtimes_{s, \tau}Q$ be an extension of $Q$ by $G$. For $[\rho]\in \widehat{G}$, let $\chi^G_\rho$ be the character of the $G$-representation $V_\rho$. For $(g_1,\ g_2)\in G\times G$, 
\begin{equation}\label{generalized_orthogonal_relation}
\sum_{[\rho]\in \widehat{G}}\sum_{q\in Q}
\chi_\rho^G(g_1^{-1})\chi^G_{q([\rho])}(g_2)= \left\{
\begin{array}{ll}
|C_H(g_1)|,& \text{if } g_1 \text{ and } g_2 \text{ are conjugate {\em in }}H,\\
0&\text{otherwise}.
\end{array}
\right.
\end{equation}\end{theorem}

\begin{proof}
 Consider
(\ref{eq:extension_reproduced}) again. The group $H\times H$ acts
naturally on the group algebra $\com H$ via
$$(h_1, h_2)\cdot h=h_1^{-1}hh_2.$$
In this way, we may view $\com H$ as a representation of $H\times
H$. Its character $\chi^{H\times H}_{\com H}$ can be calculated as follows:
\begin{equation*}
\begin{split}
\chi^{H\times H}_{\com H}((h_1, h_2))=&\#\{ h\in H| h_1^{-1}hh_2=h \}\\
=&\#\{ h\in H|hh_2h^{-1}=h_1 \}\\
=&\left\{
\begin{array}{ll}
|C_H(h_1)|,& \text{if } h_1 \text{ and } h_2 \text{ are conjugate in }H,\\
0&\text{otherwise}.
\end{array}
\right.
\end{split}
\end{equation*}

We now consider $\com H$ as a representation of the subgroup
$G\times G$. The above calculation gives the character of this
representation: for $(g_1, g_2)\in G\times G$,
\begin{equation}\label{character_answer1}
\chi^{G\times G}_{\com H}((g_1,g_2))=\chi^{H\times H}_{\com H}((g_1,g_2))= \left\{
\begin{array}{ll}
|C_H(g_1)|,& \text{if } g_1 \text{ and } g_2 \text{ are conjugate in }H,\\
0&\text{otherwise}.
\end{array}
\right.
\end{equation}
We calculate the character $\chi^{G\times G}_{\com H}$ by another method. By Proposition
\ref{prop:matrix-coeff-algebra}, there is an isomorphism of algebras
$$\com H\simeq\oplus_{[\rho]\in \widehat{G}}\End(V_{\rho})\rtimes_{T,\tau}Q.$$
Under this isomorphism, the $G\times G$ action on $\com H$ is
identified with the following $G\times G$ action on
$\oplus_{[\rho]\in \widehat{G}}\End(V_{\rho})\rtimes_{T,\tau}Q$:
\begin{equation*}
\begin{split}
(g_1, g_2)\cdot (x_\rho, q):=&(\sum_{\rho_1}\rho_1(g_1^{-1}),1)\circ (x_\rho, q)\circ(\sum_{\rho_2}\rho_2(g_2),1)\\
=&(\rho(g_1^{-1})x_\rho
{T_q^{[\rho]}}^{-1}q([\rho])(g_2)T_q^{[\rho]},q).
\end{split}
\end{equation*}
Here, $\circ$ is the algebra structure on $\oplus_{[\rho]\in
\widehat{G}}\End(V_{\rho})\rtimes_{T,\tau}Q$.

For each $\rho$, fix an isomorphism of $\End(V_\rho)$ with a matrix
algebra, and let $e_{st}^\rho$ denote the standard basis of this
matrix algebra. We use the symbol $(x_\rho)_{st}$ to denote the
$s,t$-entry of $x_\rho\in \End(V_\rho)$. Then we have
$(\rho(g_1^{-1})e_{st}^\rho
{T_q^{[\rho]}}^{-1}q([\rho])(g_2)T_q^{[\rho]})_{st}=(\rho(g_1^{-1}))_{ss}({T_q^{[\rho]}}^{-1}q([\rho])(g_2)T_q^{[\rho]})_{tt}$.
Therefore,
\begin{equation*}
\begin{split}
\text{tr}\,((g_1,g_2)|_{\End(V_\rho)\times \{q\}})=&\sum_{s,t} (\rho(g_1^{-1}))_{ss}({T_q^{[\rho]}}^{-1}q([\rho])(g_2)T_q^{[\rho]})_{tt}\\
=&\text{tr}\,(\rho(g_1^{-1}))\text{tr}\, ({T_q^{[\rho]}}^{-1}q([\rho])(g_2)T_q^{[\rho]})\\
=&\chi^G_\rho(g_1^{-1})\chi^G_{q([\rho])}(g_2),
\end{split}
\end{equation*}
where $\chi^G_\rho$ and $\chi^G_{q([\rho])}$ denote the characters of
the $G$-representations $\rho$ and $q([\rho])$. Summing over $[\rho]\in
\widehat{G}$ and $q\in Q$, we find that
\begin{equation}\label{character_answer2}
\chi^{G\times G}_{\com H}((g_1,g_2))=\sum_{[\rho]\in \widehat{G}}\sum_{q\in Q}
\chi_\rho(g_1^{-1})\chi_{q([\rho])}(g_2).
\end{equation}
Combining the above with (\ref{character_answer1}), we obtain the
desired identity:
\[
\sum_{[\rho]\in \widehat{G}}\sum_{q\in Q}
\chi^G_\rho(g_1^{-1})\chi^G_{q([\rho])}(g_2)= \left\{
\begin{array}{ll}
|C_H(g_1)|,& \text{if } g_1 \text{ and } g_2 \text{ are conjugate {\em in }}H,\\
0&\text{otherwise}.
\end{array}
\right.
\] 
\end{proof}

\end{document}